\documentclass[11pt]{article}

\usepackage[margin=1.27in]{geometry}

\usepackage[misc]{ifsym} 
\usepackage{lipsum}
\usepackage{amsmath}
\usepackage{amssymb}
\usepackage{amsfonts}
\usepackage{latexsym}
\usepackage{bm}
\usepackage{enumitem}
\usepackage{graphicx}
\usepackage{amscd}
\usepackage{extarrows}
\usepackage{indentfirst}
\usepackage{bbm}
\usepackage{mathdots}
\usepackage{mathtools}
\usepackage[linktocpage]{hyperref}
\usepackage{amsthm}
\usepackage{upgreek}
\usepackage{xcolor} 
\hypersetup{
  colorlinks,
  citecolor=black,
  filecolor=black,
  linkcolor=blue,
  urlcolor=black
}

\usepackage{quiver}
\usepackage[multiple]{footmisc}
\usepackage[title]{appendix}

\usepackage{hhline}
\usepackage{tabularx}
\usepackage{diagbox}
\usepackage{array}
\usepackage{tikz}
\usepackage{colortbl}
\usepackage{multirow}
\usepackage{bbding}
\usepackage{fontawesome5}

\usepackage{float}

\theoremstyle{plain}
\newtheorem{theorem}{Theorem}[section]
\newtheorem{proposition}[theorem]{Proposition}
\newtheorem{lemma}[theorem]{Lemma}
\newtheorem{corollary}[theorem]{Corollary}
\newtheorem*{theorem*}{Theorem}
\newtheorem*{lemma*}{Lemma}

\theoremstyle{remark}
\newtheorem{remark}[theorem]{Remark}

\theoremstyle{definition}
\newtheorem{definition}{Definition}[section]
\newtheorem{example}[theorem]{Example}

\newcommand*\dd{\mathop{}\!\mathrm{d}}

\DeclareMathOperator{\End}{End}
\DeclareMathOperator{\Ker}{Ker}

\DeclareMathOperator{\Hom}{Hom}
\DeclareMathOperator{\spn}{span}

\DeclareMathOperator{\tr}{tr}

\DeclareMathOperator{\Tr}{Tr}

\DeclareMathOperator{\Ind}{Ind}

\DeclareMathOperator{\Res}{Res}

\DeclareMathOperator{\GL}{GL}

\DeclareMathOperator{\U}{U}
\DeclareMathOperator{\GU}{GU}
\DeclareMathOperator{\nrd}{nrd}

\DeclareMathOperator{\Ot}{O}

\DeclareMathOperator{\GO}{GO}
\DeclareMathOperator{\GSp}{GSp}
\DeclareMathOperator{\JL}{JL}

\DeclareMathOperator{\Mp}{Mp}

\DeclareMathOperator{\cusp}{cusp}

\DeclareMathOperator{\fin}{fin}

\DeclareMathOperator{\disc}{disc}

\DeclareMathOperator{\irr}{Irrep}

\DeclareMathOperator{\Sp}{Sp}

\title{The Global Jacquet--Langlands Correspondence via Tensor Products}
\author{Jun Yang}
\date{}

\begin{document}
\maketitle

\begin{abstract}
We prove that the global Jacquet--Langlands correspondence $\JL$ for $\GL(2)$ can be realized via tensor products over Hecke algebras.
Let $G$ be a non-split inner form of $\GL(2)$ over a number field.
Using the similitude theta correspondence, the space $L^2(D(\mathbb{A})\times \mathbb{A}^{\times})$ acquires the structure of a $G(\mathbb{A})$--$(G(\mathbb{A})\times \GL(2,\mathbb{A}))$ bimodule such that
\begin{equation*}
L^2(G(F)\backslash G(\mathbb{A}),\chi) \otimes_{\mathcal{H}(G)} L^2(D(\mathbb{A})\times \mathbb{A}^{\times})
\;\cong\;
\bigoplus_{\pi \in \irr_{\mathcal{A}}(G,\chi^{-1})} \pi \otimes \JL(\pi). 
\end{equation*}
This decomposition into irreducible representations of $G(\mathbb{A})\times \GL(2,\mathbb{A})$ recovers the full global Jacquet--Langlands correspondence.
\end{abstract}

\tableofcontents

\section{Introduction}

The theta correspondence and Howe duality provide one of the most powerful mechanisms for transferring representations between different reductive groups over a local field.
Given a dual pair $(G,H)$ inside a symplectic group, the Weil representation realizes a correspondence between irreducible admissible representations of $G$ and $H$, known as the \emph{theta correspondence}.
\emph{Howe duality} asserts that, under suitable hypotheses, the small theta lift establishes a bijection between appropriate subsets of the admissible duals of $G$ and $H$.
Since its inception, Howe duality has played a central role in representation theory and automorphic forms, serving as a fundamental source of instances of Langlands functoriality.

The classical theory of the theta correspondence usually concerns isometry groups.
However, many arithmetic applications \cite{HarrisKudla1992}, especially those involving automorphic representations and global transfer problems, naturally require passing to \emph{similitude groups}.
To address this, Roberts \cite{RobertBrooks1996} and later Gan--Takeda \cite{GanTak2011}, Gan--Tantono \cite{GanTan2014}, and Baki\'{c}--Gan--Savin \cite{BGS2024} developed the theory of the \emph{induced Weil representation}, which extends the theta correspondence to similitude dual pairs.
In this framework, Howe duality continues to hold, and the resulting similitude theta correspondence exhibits strong local and global compatibility properties.

The guiding philosophy of this paper is that the \emph{Jacquet--Langlands correspondence} \cite{JL70} for $\GL(2)$ should be viewed as a concrete manifestation of Howe duality for a specific quaternionic similitude dual pair.
More precisely, we show that the global Jacquet--Langlands transfer arises naturally from the theta correspondence for the dual pair $D^{\times}$ and $(D^{\times}\times \GL(2))/F^{\times}$,
where $D$ is a quaternion division algebra over a number field $F$.
In this setting, the representation-theoretic content of the Jacquet--Langlands correspondence is encoded in the structure of the induced Weil representation and its behavior as a module over suitable Hecke algebras.
Therefore, the central technical theme of this paper is the realization of theta lifts through tensor products over certain algebras, inspired by the work of Goffeng--Mesland--Sengun on local/global group $C^*$-algebras \cite{GMS2024} and the work of Loke--Przebinda on local Hecke algebras \cite{LokePrze2024}.

Let $G=D^{\times}$ and define the global Hecke algebra $\mathcal{H}(G)$ of $G$ by the restricted tensor product over the characteristic functions of maximal compact subgroups (at finite places), i.e.,
$\mathcal{H}(G)=\otimes_v{}' \mathcal{H}(G_v)$.
Here $\mathcal{H}(G_v)$ is the usual Hecke algebra in the non-archimedean case, and the convolution algebra $\mathcal{H}(\mathfrak{g}_v,K_v)$ in the archimedean case.

Let $\Omega_{\mathbb{A}}$ be the global induced Weil representation associated with the above similitude dual pair.
We show that for an irreducible admissible representation $\pi$ of $G(\mathbb{A})$, the similitude big theta lift of $\pi$ can be realized as the tensor product of modules over the global Hecke algebra $\mathcal{H}(G)$:
\[
\Theta(\pi) \;\cong\; \pi^{\vee} \otimes_{\mathcal{H}(G)} \Omega_{\mathbb{A}}.
\]
This tensor description holds both locally and globally, provided Howe duality holds at every place.

Our main result applies this formalism to the global Jacquet--Langlands correspondence for $\GL(2)$.
Let $\pi$ be an irreducible automorphic representation of $D^{\times}(\mathbb{A})$ with central character $\chi$, and let $\JL(\pi)$ denote its Jacquet--Langlands transfer to $\GL(2,\mathbb{A})$.
Using the irreducibility of local big theta lifts for the quaternionic similitude dual pair and the local--global compatibility of the similitude theta correspondence, we prove that there is an isomorphism of representations
\[
\pi^{\vee} \otimes_{\mathcal{H}(G)} \Omega_{\mathbb{A}}
\;\cong\;
\pi \otimes \JL(\pi)
\]
of $D^{\times}(\mathbb{A}) \times \GL(2,\mathbb{A})$.
Equivalently, the Jacquet--Langlands correspondence is realized by a universal tensor construction governed by Howe duality. 
As a consequence, the correspondence appears naturally at the level of automorphic spectra:
\[
L^2(G(F)\backslash G(\mathbb{A}),\chi) \otimes_{\mathcal{H}(G)} L^2(D(\mathbb{A})\times \mathbb{A}^{\times})
\;\cong\;
\bigoplus_{\pi \in \irr_{\mathcal{A}}(G,\chi^{-1})} \pi \otimes \JL(\pi),
\]
where $\irr_{\mathcal{A}}(G,\chi^{-1})$ denotes the automorphic representation of $G(\mathbb{A})$ with the central character $\chi^{-1}$. 
This perspective highlights the Jacquet--Langlands correspondence as an instance of theta functoriality, rather than an isolated phenomenon.

Beyond providing a new conceptual proof of the Jacquet--Langlands correspondence for $\GL(2)$, the results of this paper emphasize the structural role of Howe duality in global representation theory.
The tensor-product realization presented here avoids trace formulas and treats archimedean and non-archimedean places uniformly.
It also suggests a general strategy for realizing other instances of functorial transfer arising from theta correspondences via Hecke-module tensor products.

\medskip
\noindent
\textbf{Organization of the paper.}
Section~\ref{sltheta} reviews Howe duality and the classical theory of the theta correspondence.
In Section~\ref{slocthetasim}, we summarize the theta correspondence for similitude groups via the induced Weil representation.
Section~\ref{stentheta} establishes the tensor-product realization of local and global similitude theta lifts.
Finally, Section~\ref{stenJL} applies this framework to obtain a tensor realization of the global Jacquet--Langlands correspondence for $\GL(2)$.

\medskip
\noindent
\textbf{Acknowledgments.} The author would like to thank Mehmet Haluk \c{S}eng\"un for his insightful comments during our discussions and collaboration.

\section{Howe duality and theta correspondences}\label{sltheta}

We first recall the formal statement of Howe duality.
Let $G,H$ be a pair of reductive groups over a local field $F$.
Let $(\rho,S)$ be a smooth representation of $G\times H$.
For $\pi\in \irr(G)$ (the set of equivalence classes of irreducible admissible representations of $G$), we set
\begin{equation}\label{eSpi}
    S(\pi)=S/\bigcap\limits_{T\in \Hom_{G}(S,\pi)}\Ker(T),
\end{equation}
which is a smooth representation of $G\times H$.
There exists a smooth $H$-representation $\Theta(\pi)$ (the \emph{big theta lift} of $\pi$) such that
\begin{equation}\label{eHowe}
S(\pi)\cong \pi\otimes\Theta(\pi),
\end{equation}
when $(\rho,S)$ is the Weil representation (for archimedean $F$, see \cite[Section~2]{Howe89}; for non-archimedean $F$, see \cite[Lemma~III.3]{MVW87}), or the induced Weil representation for similitude groups (see \cite[Section~3]{BGS2024}).
We define two families of irreducible representations of $G$ and $H$:
\begin{itemize}
\item $\mathcal{R}(G)=\{\pi\in\irr(G)\mid S(\pi)\neq 0\}$ and $\mathcal{R}(H)=\{\tau\in\irr(H)\mid S(\tau)\neq 0\}$.
\end{itemize}
Note that $\Theta(\pi)$ admits a maximal semisimple quotient $\theta(\pi)$ (the \emph{small theta lift} of $\pi$).
We say that \emph{Howe duality} holds if the graph
\[
\mathcal{R}(G\times H)=\{(\pi,\tau)\in \mathcal{R}(G)\times \mathcal{R}(H)\mid \Hom_{G\times H}(\rho,\pi\otimes\tau)\neq 0\}
\]
defines a bijection between $\mathcal{R}(G)$ and $\mathcal{R}(H)$ given by $\pi \mapsto \theta(\pi)$.

We will consider theta correspondences where $(\rho,S)$ is taken to be the Weil representation.
We mainly follow \cite{PrasadD1993,Gan2024} for this review.
Let $F$ be a local field and $W$ be a symplectic space over $F$ equipped with a symplectic form $\langle~,~\rangle$.
Let $H(W)$ be the \emph{Heisenberg group}, defined on the set $W\times F$ with group multiplication
\[
(w_1,z_1)\cdot (w_2,z_2)=(w_1+w_2,z_1+z_2+\tfrac{1}{2}\langle w_1,w_2\rangle).
\]
Let $W=W_1\oplus W_2$ be a decomposition into maximal isotropic subspaces.

For a nontrivial unitary additive character $\psi\colon F\to \U(1)$, there is a unique (up to equivalence) irreducible representation $\rho_{\psi}$ of $H(W)$ with central character $\psi$, called the \emph{Heisenberg representation}, which can be realized as follows.
Let $\mathcal{S}=\mathcal{S}(W_1)$ be the space of smooth functions (archimedean $F$) or locally constant compactly supported functions (non-archimedean $F$) on $W_1$.
The representation $\rho_{\psi}$ acts on $\mathcal{S}$ by
\begin{equation*}
\begin{aligned}
\rho_{\psi}((x,0,0))\phi(w_1)&=\phi(w_1+x);\\
\rho_{\psi}((0,y,0))\phi(w_1)&=\psi(\langle w_1,y\rangle)\phi(w_1);\\
\rho_{\psi}((0,0,a))\phi(w_1)&=\psi(a)\phi(w_1).
\end{aligned}
\end{equation*}
This is called the \emph{Schr\"odinger model} of the Heisenberg representation. 
The symplectic group
\[
\Sp(W)=\{g\in \GL(W)\mid \langle gw_1,gw_2\rangle=\langle w_1,w_2\rangle \text{ for all } w_1,w_2\in W\}
\]
acts on $H(W)$ by $g\cdot (w,t)=(gw,t)$.
Thus, there exists an operator $\omega_{\psi}(g)$ (well-defined up to a scalar) on $\mathcal{S}$ such that
\begin{equation}\label{emetagrp}
\rho_{\psi}(gw,t)\cdot \omega_{\psi}(g)=\omega_{\psi}(g)\cdot \rho_{\psi}(w,t)
\end{equation}
for $(w,t)\in H(W)$.
The \emph{metaplectic group} is defined by
\[
\widetilde{\Sp}_{\psi}(W)=\{(g,\omega_{\psi}(g))\mid \rho_{\psi}(gw,t)\,\omega_{\psi}(g)=\omega_{\psi}(g)\,\rho_{\psi}(w,t),\ \forall (w,t)\in H(W)\},
\]
which fits into the exact sequence $1\to \mathbb{C}^{\times}\to \widetilde{\Sp}_{\psi}(W)\to \Sp(W)\to 1$.
The representation
\[
\omega_{\psi}\colon (g,\omega_{\psi}(g))\mapsto \omega_{\psi}(g)\in \GL(\mathcal{S})
\]
is called the \emph{Weil representation} (or the \emph{oscillator representation}).

Given a pair of reductive subgroups $(G,G')$ of $\Sp(W)$, suppose $G$ and $G'$ are mutual centralizers in $\Sp(W)$; we then call them a \emph{dual pair}.
For an irreducible representation $(\pi,V_{\pi})$ of $G$ (resp.\ $G'$), the representation $\omega_{\psi}$ has a maximal $\pi$-isotypic quotient $\mathcal{S}(\pi)$.
As in \eqref{eHowe}, it is known that
\begin{equation}\label{eThetaQuoLocal}
\mathcal{S}(\pi)\cong \pi\otimes \Theta_{\psi}(\pi)
\end{equation}
for some smooth representation $\Theta_{\psi}(\pi)$ of $G'$.
Thus, we obtain the big theta lift $\Theta_{\psi}(\pi)$ of $\pi$ and its maximal semisimple quotient $\theta_{\psi}(\pi)$ (the small theta lift of $\pi$).

\begin{theorem}
Howe duality holds for the Weil representation $(\rho,\mathcal{S})$.
\end{theorem}
This result was proved in the archimedean case by Howe \cite{Howe77}.
The non-archimedean case was proved by Waldspurger \cite{Waldspurger1989} when the residue characteristic of $F$ is not $2$, and later in full generality by Gan and Takeda \cite{GanTak2016} and Gan and Sun \cite{GanSun2017}.

Now let $F$ be a number field and consider global theta lifts.
Let $\mathcal{O}$ be the ring of integers of $F$.
Let $W$ be a symplectic space over $F$ with a maximal isotropic decomposition $W=W_1\oplus W_2$.
For each place $v$ of $F$, let $F_v$ be the completion and $\mathcal{O}_v$ the ring of integers of $F_v$ (for non-archimedean $v$).
Fix a standard symplectic basis $\{e_1,\dots,e_n,f_1,\dots,f_n\}$ of $W$, and let
\[
L=\langle e_1,\dots,e_n,f_1,\dots,f_n\rangle_{\mathcal{O}}
\]
be the $\mathcal{O}$-lattice generated by this basis.

Let $W_v=W\otimes_F F_v$ and $L_v=L\otimes_{\mathcal{O}}\mathcal{O}_v$.
Define $H^{o}(W_v)=L_v\oplus \mathcal{O}_v$, a compact open subgroup of $H(W_v)$.
The adelic Heisenberg group is the restricted product $H(W_{\mathbb{A}})=\prod'(H(W_v),H^{o}(W_v))$.
Let $J_v$ be the stabilizer of $H^{o}(W_v)$ inside $\Sp(W_v)$.
The adelic symplectic group $\Sp(W_{\mathbb{A}})$ is the restricted product $\prod'(\Sp(W_v);J_v)$.
Fix a global additive character $\psi\colon \mathbb{A}/F\to \U(1)$ and define the adelic metaplectic group $\widetilde{\Sp}_{\psi}(W_{\mathbb{A}})$ similarly to \eqref{emetagrp}, which satisfies
\[
1\to \U(1)\to \widetilde{\Sp}_{\psi}(W_{\mathbb{A}}) \to \Sp(W_{\mathbb{A}})\to 1.
\]
We mention that $\widetilde{\Sp}_{\psi}(W_{\mathbb{A}})$ is a quotient of the restricted product $\prod'(\widetilde{\Sp}_{\psi_v}(W_v); J_v)$ by a central subgroup.
Let
\[
\omega_{\psi}\colon \widetilde{\Sp}_{\psi}(W_{\mathbb{A}}) \to \GL(\mathcal{S}(W_{1,\mathbb{A}}))
\]
be the \emph{global Weil representation}.

Let $(G(\mathbb{A}),H(\mathbb{A}))$ be a dual pair in $\widetilde{\Sp}_{\psi}(W_{\mathbb{A}})$.
The \emph{global theta correspondence} can be constructed as follows.
For $\varphi\in \mathcal{S}(W_{1,\mathbb{A}})$, we define the theta function attached to $\varphi$ by
\[
\theta(\varphi)(g)=\sum_{x\in W_{1}(F)}(\omega_{\psi}(g)\varphi)(x),
\]
which is left $\Sp(W)(F)$-invariant on $\Sp(W_{\mathbb{A}})$.
Assume $\pi$ is a cuspidal automorphic representation of $G(\mathbb{A})$, i.e.\ $(\pi,V_{\pi})$ is an irreducible subrepresentation of $G(\mathbb{A})$ contained in $L^2_{\cusp}(G(F)\backslash G(\mathbb{A}))$.
For $f\in V_{\pi}$ and $\varphi\in \mathcal{S}(W_{1,\mathbb{A}})$, define the \emph{global theta lift} associated with $f$ and $\varphi$ by
\[
\theta(\varphi,f)(h)=\int_{G(F)\backslash G(\mathbb{A})}\theta(\varphi)(hg)\overline{f(g)}\,\dd g
\]
for $h\in H(\mathbb{A})$.
Define
\[
V(\pi)=\spn\{\theta(\varphi,f)\mid f\in V_{\pi},\ \varphi\in  \mathcal{S}(W_{1,\mathbb{A}})\},
\]
an automorphic representation of $H(\mathbb{A})$, called the \emph{global theta lift} of $\pi$.

\begin{theorem}\label{tlgcomp}
Assume local Howe duality holds for the dual pair $(G,H)$ at every place of $F$.
Let $\pi=\otimes'_v\pi_v$ be a cuspidal automorphic representation of $G(\mathbb{A})$ and write $\theta(\pi)=\otimes_v \theta(\pi_v)$.
If the global theta lift $V(\pi)$ is nonzero and is contained in $L^2(H(F)\backslash H(\mathbb{A}))$, then
\[
V(\pi)\cong \theta(\pi).
\]
\end{theorem}
This gives the compatibility of the local and global theta correspondences.
See \cite[Proposition~3.1]{Gan2024} for a proof.

\section{Theta correspondences for similitude groups}\label{slocthetasim}

We follow \cite{BGS2024} for a general description of the similitude theta correspondence.
Then we give two examples: symplectic--orthogonal dual pairs as studied by Roberts \cite{RobertBrooks1996}, and quaternionic dual pairs as studied by Gan and Tantono \cite{GanTan2014}.

Suppose $(G,H)$ is a pair of reductive subgroups inside a simple linear algebraic group $\mathcal{E}$, all defined over a local field $F$.
Assume that $(G,H)$ forms a dual pair in the sense that they are mutual centralizers of each other.
Fix a homomorphism $i\colon G\times H\to\mathcal{E}$, which is not necessarily injective.
Let $Z\hookrightarrow G$ and $Z\hookrightarrow H$ be embeddings of a common central subgroup $Z$ (identified under $i$) such that
\begin{itemize}
    \item $\ker(i)=\Delta(Z):=\{(z,z^{-1})\mid z\in Z\}$,
    \item $(G\times H)/\Delta(Z)\to \mathcal{E}$ is injective.
\end{itemize}
Then the relevant similitude groups can be defined as follows.

\begin{definition}
The \emph{similitude group} $\widetilde{G}$ associated to $G$ and the following data:
\begin{enumerate}
    \item an embedding $j\colon Z\hookrightarrow T:=\Res_{K/F}(\mathbb{G}_m)$ for an \'{e}tale $F$-algebra $K$ of finite rank;
    \item an embedding $T/j(Z)\hookrightarrow \Res_{E/F}(\mathbb{G}_m)$ for some finite extension $E/F$,
\end{enumerate}
is defined by
\begin{center}
    $\widetilde{G}:=(G\times T)/\Delta(Z)$.
\end{center}
It comes with the similitude homomorphism
\begin{center}
    $\lambda_G\colon \widetilde{G}\twoheadrightarrow T/j(Z)\hookrightarrow \Res_{E/F}(\mathbb{G}_m)$.
\end{center}
\end{definition}
\noindent The similitude group $\widetilde{H}$ and the similitude homomorphism $\lambda_{H}$ are defined similarly.

Consider the diagonal map $\Delta\colon T\to \widetilde{G}\times \widetilde{H}$.
We use the following notation:
\begin{itemize}
\item $\omega$ = a minimal representation of $\mathcal{E}(F)$ (see \cite{GanSav2005});
\item $R_0:=\{(g,h)\in \widetilde{G}\times \widetilde{H}\mid\lambda_{G}(g)\cdot \lambda_{H}(h)=1\}$;
\item $R_1:=R_0/\Delta(T)$, which satisfies $R_1\cong (G\times H)/\Delta(Z)$;
\item $\widetilde{G}(F)^+$ := the image of the projection $R_0\to \widetilde{G}(F)$;
\item $\widetilde{H}(F)^+$ := the image of the projection $R_0\to \widetilde{H}(F)$.
\end{itemize}
It is known that $\widetilde{G}(F)^+=\widetilde{G}(F)$ if and only if $\lambda_G(\widetilde{G}(F))\subset \lambda_H(\widetilde{H}(F))$, which holds if $\lambda_H$ is surjective.
We call the pair $(\widetilde{G}(F)^+,\widetilde{H}(F)^+)$ a \emph{similitude dual pair}.

Consider the composition of homomorphisms
\begin{equation}\label{eextWeilgrp}
    R_0\twoheadrightarrow R_1 \hookrightarrow \mathcal{E}(F).
\end{equation}
For the $\mathcal{E}(F)$-representation $\omega$, its pullback along \eqref{eextWeilgrp} gives a representation of $R_0$.
We denote this $R_0$-representation by $\overline{\omega}$, which plays the role of the extended Weil representation in \cite[Section~3]{RobertBrooks1996} and \cite[Section~3]{GanTan2014}.
We then define the induced representation
\begin{equation}\label{eindWeil}
    \Omega=\Ind_{R_0}^{\widetilde{G}(F)^{+}\times\widetilde{H}(F)^{+}}\overline{\omega}.
\end{equation}
For $\pi\in \irr(\widetilde{G}(F)^{+})$, we define
\[
\Omega(\pi)=\Omega\Big/\bigcap\limits_{T\in \Hom_{\widetilde{G}(F)^+}(\Omega,\pi)}\Ker(T).
\]
as in \eqref{eSpi}. 
Then $\Omega(\pi)$ is a smooth representation of $\widetilde{G}(F)^{+}\times\widetilde{H}(F)^{+}$.
There exists a smooth $\widetilde{H}(F)^+$-representation $\Theta(\pi)$ such that $\Omega(\pi)\cong \pi\otimes\Theta(\pi)$; we call $\Theta(\pi)$ the \emph{similitude big theta lift} of $\pi$.
We denote its maximal semisimple quotient by $\theta(\pi)$ and call it the \emph{similitude small theta lift} of $\pi$.
It is known that $\Theta(\pi)\cong \Hom_{G(F)}(\omega,\pi)$ as a smooth representation of $\widetilde{H}(F)^{+}$.

The following result was first proved by Roberts (see \cite[Theorem~4.4]{RobertBrooks1996}) for symplectic--orthogonal dual pairs under the assumption that the restrictions to isometry groups are multiplicity-free, then for the quaternionic cases by Gan and Tantono (see \cite[Proposition~3.3(iv)]{GanTan2014}), and in general by Baki\'{c}, Gan, and Savin (see \cite[Proposition~5.1]{BGS2024}).

\begin{theorem}\label{tsimHowe}
If Howe duality holds for a dual pair $(G(F),H(F))$ over a local field $F$, then Howe duality holds for the associated similitude pair $(\widetilde{G}(F)^+,\widetilde{H}(F)^+)$.
\end{theorem}

Here we give two examples where the minimal representation $\omega$ is taken to be the Weil representation.

\begin{example}[symplectic--orthogonal dual pairs]
This is the example studied by Roberts in \cite{RobertBrooks1996}.
We fix a non-archimedean local field $F$.
Let $X$ be a finite-dimensional nondegenerate symplectic space over $F$ with symplectic form $(~,~)_X$, and let $Y$ be a finite-dimensional nondegenerate symmetric space over $F$ with inner product $(~,~)_Y$.
For simplicity, we assume that $\dim_{F}X$ is even, while the odd-dimensional case was also studied.

The space $W=X\otimes Y$ is symplectic with symplectic form
$\langle x_1\otimes y_1,x_2\otimes y_2\rangle=(x_1,x_2)_X\cdot (y_1,y_2)_Y$.
The isometry groups $(\U(X),\U(Y))=(\Sp(X),\Ot(Y))$ form a dual pair in $\Sp(W)$.
In this case, $\ker(i)=\Delta(\mu_2)$, where $\mu_2$ is the group of second roots of unity.
Assume that $X=X_1\oplus X_2$ is a decomposition into maximal isotropic subspaces.
Then $W=W_1\oplus W_2$ with $W_i=X_i\otimes Y$ giving a maximal isotropic decomposition.
The similitude groups are
\begin{center}
    $\widetilde{G}(F)^{+}=\GSp(X)$ and $\widetilde{H}(F)^{+}=\GO(Y)$.
\end{center}

The construction of the extended/induced Weil representation relies on lifting the map
$i\colon \Sp(X)\times\Ot(Y)\to\Sp(X\otimes Y)$ to the metaplectic cover $\Mp(X\otimes Y)$.
This lifting can then be extended to
$R_0=\{(g,h)\in \widetilde{G}\times \widetilde{H}\mid\lambda_{G}(g)\cdot \lambda_{H}(h)=1\}$
to form the extended Weil representation $\overline{\omega}$ of $R_0$, and the induced Weil representation $\Omega$ follows.
Moreover, if we denote the underlying space of $\Omega$ by $\mathcal{T}$, then there is a $\widetilde{G}(F)^{+}\times\widetilde{H}(F)^{+}$-equivariant isomorphism from $\mathcal{T}$ to
$\tilde{\mathcal{S}}=\mathcal{S}(W_1\times F^{\times})$, the space of Schwartz--Bruhat functions on $W_1\times F^{\times}$.
\end{example}

\begin{example}[quaternionic dual pairs]\label{eGTquasim}
Here is a family of quaternionic spaces together with associated isometry and similitude groups, constructed by Gan and Tantono in \cite[Section~2]{GanTan2014}.
Fix a quaternion division algebra $D$ over a non-archimedean field $F$.
\begin{enumerate}
\item For $k\geq 1$, let $V_k=D^k$ be the unique quaternionic Hermitian space with $\dim_{D}V_{k}=k$.
More precisely, $V_1=D$ and $V_2=D e_{1}\oplus D e_{2}$ with inner product given by $\langle e_1,e_2\rangle=1$ and $\langle e_i,e_i\rangle=0$.
Then the Witt towers of these quaternionic Hermitian spaces are
\[
V_{2n}=V_{2}^{\oplus n}\quad \text{and}\quad V_{2n+1}=V_{1}\oplus V_{2}^{\oplus n}.
\]
Thus, the isometry groups are
\[
\U(V_k)=\begin{cases}
\U(V_{2n})=\Sp_{n,n}, &\text{ if } k=2n;\\
\U(V_{2n+1})=\Sp_{n+1,n}, &\text{ if } k=2n+1,
\end{cases}
\]
which are inner forms of $\Sp_{2k}$.
The similitude groups are
\[
\GU(V_k)=\begin{cases}
\GU(V_{2n})=\GSp_{n,n}, &\text{ if } k=2n;\\
\GU(V_{2n+1})=\GSp_{n+1,n}, &\text{ if } k=2n+1.
\end{cases}
\]

\item For $k\geq 2$ and a separable quadratic $F$-algebra $K$ (identified with an element of $F^{\times}/F^{\times2}$), let $W_k^K$ be the unique quaternionic skew-Hermitian space of dimension $k$ and discriminant $K$.
The space $W_2=D f_1\oplus D f_2$ is equipped with $(f_1,f_2)=1$ and $(f_i,f_i)=0$.

For $i=1,2,3$, let $a_i\in F^{\times}/F^{\times2}$ satisfy $a_1 a_2 a_3=1$ in $F^{\times}/F^{\times2}$.
Let $W_3\cong D f_1\oplus D f_2\oplus D f_3$ such that $(f_i,f_i)=a_i$.
For $m\geq 1$, define the skew-Hermitian spaces
\[
W_{2m}=W_2^{\oplus m}\quad \text{and}\quad W_{2m+3}=W_3\oplus W_2^{\oplus m}.
\]
Thus, the isometry groups are
\[
\U(W_l)=\begin{cases}
\U(W_{2m})=\Ot^{m,m}, &\text{ if } l=2m;\\
\U(W_{2m+3})=\Ot^{m+3,m}, &\text{ if } l=2m+3,
\end{cases}
\]
which are inner forms of $\Sp_{2k}$.
The similitude groups are
\[
\GU(W_l)=\begin{cases}
\GU(W_{2m})=\GO^{m,m}, &\text{ if } l=2m;\\
\GU(W_{2m+3})=\GO_{m+3,m}, &\text{ if } l=2m+3.
\end{cases}
\]
\end{enumerate}

Note that $V_k$ is Hermitian and $W_l$ is skew-Hermitian.
Their tensor product $V_{k}\otimes_{D} W_{l}$ is symplectic with symplectic form
\[
\langle v_1\otimes w_1,v_2\otimes w_2\rangle=\tfrac{1}{2}\Tr_{D/F}(\langle v_1,v_2\rangle\cdot\overline{(w_1,w_2)}),
\]
where $v_1,v_2\in V_k$, $w_1,w_2\in W_l$, and $\Tr_{D/F}(x)=x+\overline{x}$ is the canonical trace on $D$.
We obtain a natural map
\[
\U(V_k)\times\U(W_l)\to \Sp(V_{k}\otimes_{D} W_{l}),
\]
which makes $(\U(V_k),\U(W_l))$ a dual pair.
This also induces a natural map
\[
\GU(V_k)\times\GU(W_l)\to \GSp(V_{k}\otimes_{D} W_{l}).
\]

Let $(V,W)$ be of the form $(V_k,W_l)$ above, and let $(\U(V),\U(W))$ be a dual pair of this type.
Then $(\GU(V),\GU(W))$ forms the associated similitude dual pair.
\end{example}

The induced representation $\Omega$ can also be described by an induced representation depending only on $G(F)$ and its associated similitude group.

\begin{proposition}\label{ptwoOmega}
We have $\Omega\cong\Ind_{G(F)}^{\tilde{G}(F)^{+}}(\omega|_{G(F)})$.
\end{proposition}
\begin{proof}
For any irreducible representation $\tilde{\pi}$ of $\tilde{G}(F)^{+}$, it is known that
\[
\Hom_{G(F)}(\omega,\Res_{G(F)}^{\tilde{G}(F)^{+}}\tilde{\pi})
\;\cong\;
\Hom_{\tilde{G}(F)^{+}}(\Omega,\tilde{\pi})
\]
by \cite[\S 3]{BGS2024}.
By Frobenius reciprocity, we obtain $\Omega\cong\Ind_{G(F)}^{\tilde{G}(F)^{+}}\omega$.
\end{proof}

\section{The tensor product for similitude theta correspondence}\label{stentheta}

We realize both local and global similitude theta lifts via tensor products over suitable Hecke algebras.
From this section on, we write $(G,H)$ for a similitude dual pair when there is no risk of confusion with the ordinary dual pair.

Suppose $F$ is a non-archimedean local field and $G$ is a reductive group over $F$.
Let $C_c^{\infty}(G)$ be the complex vector space of compactly supported smooth functions on $G$.
Recall that the \emph{Hecke algebra} $\mathcal{H}(G)$ is defined to be space $C_c^{\infty}(G)$ together with multiplication given by convolution:
\[
(f_1*f_2)(g)=\int_G f_1(x)\,f_2(x^{-1}g)\,\dd x
\]
for $f_1,f_2\in C_c^{\infty}(G)$.
Given a representation $(\pi,V)$ of $G$, define
\[
\pi(f)v=\int_G f(g)\,\pi(g)v\,\dd g
\]
for $f\in \mathcal{H}(G)$.
This makes $V$ into an $\mathcal{H}(G)$-module.
Thus, the category of smooth representations of $G$ is equivalent to the category of nondegenerate $\mathcal{H}(G)$-modules\footnote{An $\mathcal{H}(G)$-module $M$ is called \emph{nondegenerate} if for every $m\in M$ there exist $n\ge 1$ and $(x_i,m_i)_{1\le i\le n}$ in $\mathcal{H}(G)\times M$ such that $m=\sum_{i=1}^{n}x_{i}m_{i}$.} (see \cite[Lemma~5.3.3]{GetzHahn}).

Now consider the archimedean case, i.e.\ $F=\mathbb{R}$ or $F=\mathbb{C}$.
Let $G$ be a real Lie group with complexified Lie algebra $\mathfrak{g}$ and maximal compact subgroup $K$.
For admissible representations of $G$, we work with the underlying $(\mathfrak{g},K)$-modules, which preserve irreducibility (see \cite[\S 3.4]{Wal}).
Let $\mathcal{H}(\mathfrak{g},K)$ be the convolution algebra of $K$-finite distributions on $G$ supported in $K$.
By \cite[Chapter~I.4, Theorem]{KnVo1995}, for $(\mathfrak{g},K)$-modules $M$ and $N$ we have
\[
\Hom_{(\mathfrak{g},K)}(M,N)=\Hom_{\mathcal{H}(\mathfrak{g},K)}(M,N).
\]
Thus, it suffices to consider modules over $\mathcal{H}(\mathfrak{g},K)$.

The following result is the similitude analogue of the ordinary Weil representation case (see \cite[Theorems~3--4]{LokePrze2024}).
\begin{lemma}\label{ltenHsim}
Let $(G,H)$ be a similitude dual pair over a local field $F$, and let $\Omega$ be the induced representation of $G\times H$.
For an irreducible representation $\pi$ of $G$, we have
\[
\Theta(\pi)\cong\pi^{\vee}\otimes_{\mathcal{H}(G)}\Omega
\quad \text{(resp.\ } \Theta(\pi)\cong\pi^{\vee}\otimes_{\mathcal{H}(\mathfrak{g},K)}\Omega\text{)}
\]
as representations of $H$, according to whether $F$ is non-archimedean (resp.\ archimedean).
\end{lemma}

\begin{proof}
We write $\mathcal{H}$ for $\mathcal{H}(G)$ (resp.\ $\mathcal{H}(\mathfrak{g},K)$).
Let $\mathcal{T}$ be the underlying space of $\Omega$.
Recall that
\[
\mathcal{T}(\pi)=\bigcap_{T\in \Hom_{G}(\mathcal{T},V_{\pi})}\Ker(T).
\]
Consider the short exact sequence
\[
0\to \mathcal{T}(\pi)\to \mathcal{T} \to V_{\pi}\otimes_{\mathbb{C}} V_{\Theta(\pi)}\to 0.
\]
Tensoring $V_{\pi}^{\vee}$ over $\mathcal{H}$ yields the exact sequence
\[
V_{\pi}^{\vee}\otimes_{\mathcal{H}}\mathcal{T}(\pi)\to V_{\pi}^{\vee}\otimes_{\mathcal{H}}\mathcal{T} \to V_{\pi}^{\vee}\otimes_{\mathcal{H}}V_{\pi}\otimes_{\mathbb{C}} V_{\Theta(\pi)}\to 0.
\]
Since $V_{\pi}^{\vee}\otimes_{\mathcal{H}}V_{\pi}\cong \Hom_{G}(V_{\pi},V_{\pi}^{\vee})\cong\mathbb{C}$, we obtain
\[
V_{\pi}^{\vee}\otimes_{\mathcal{H}}\mathcal{T}(\pi)\to V_{\pi}^{\vee}\otimes_{\mathcal{H}}\mathcal{T} \to  V_{\Theta(\pi)}\to 0.
\]
If $V_{\pi}^{\vee}\otimes_{\mathcal{H}}\mathcal{T}(\pi)\neq 0$, then
$V_{\pi}\subset \mathcal{T}(\pi)$ since
$V_{\pi}^{\vee}\otimes_{\mathcal{H}}\mathcal{T}(\pi)\subset \Hom_{\mathcal{H}}(V_{\pi},\mathcal{T}(\pi))$,
which contradicts the definition of $\mathcal{T}(\pi)$.
Hence $\Theta(\pi)\cong\pi^{\vee}\otimes_{\mathcal{H}}\Omega$.
\end{proof}

For a non-archimedean place $v$, let $K_v=G(\mathcal{O}_v)$ denote the maximal compact subgroup of $G(F_v)$.
Let $\chi_v$ be the characteristic function of $K_v$ (up to a suitable normalization of Haar measure).
Thus, for an $F$-group $G$, we define the Hecke algebra for finite adeles as the restricted tensor product
\begin{equation*}
    \mathcal{H}(G(\mathbb{A}_{\fin})):=\bigotimes_{v\in V_{\fin}}{}'(\mathcal{H}(G(F_v)),\chi_v),
\end{equation*}
where the restricted tensor product is taken over the set $V_{\fin}$ of finite places of $F$.
Following \cite[Section~7]{Kn97llds}, we define the global Hecke algebra of the adelic group $G(\mathbb{A})$ by
\begin{equation}\label{egHecke}
\mathcal{H}(G):=\mathcal{H}(G(\mathbb{A}))=\mathcal{H}(G(\mathbb{A}_{\fin}))\otimes\left(\otimes_{v\in V_{\infty}}\mathcal{H}(\mathfrak{g_v},K_v)\right),
\end{equation}
where $V_{\infty}$ denotes the set of infinite places of $F$. 
It is known that the smooth vectors of a $G(\mathbb{A})$-module form a representation of $\mathcal{H}(G)$. 

Next, we construct the underlying space of the adelic induced Weil representation.
For a non-archimedean place $v$, the induced Weil representation can be realized on $\mathcal{S}(W_{1,F_v}\times F_v^{\times})$ as shown for the symplectic-orthogonal cases in \cite[Section 3]{RobertBrooks1996}.
Fix an $F_v$-basis $\{e_1,\cdots,e_l\}$ of $W_{1,F_v}$.
Let $\eta_v$ be the characteristic function of $\langle e_1,\cdots,e_l\rangle_{\mathcal{O}_v}\times \mathcal{O}_v^{\times}$.
We then define the restricted tensor product with respect to these vectors by
\begin{equation}\label{eglobalOmega}
    \Omega_{\mathbb{A}}=\bigotimes_{v\in V_{\fin}}{}'(\Omega_v,\eta_v),
\end{equation}
where $\Omega_v$ is the induced Weil representation at $v$.
One can show that $\Omega_{\mathbb{A}}\cong \mathcal{S}(W_{1,\mathbb{A}}\times \mathbb{A}^{\times})$.
Such a construction is explicit for the ordinary global Weil representation (see \cite[\S 6]{GMS2024} and \cite[\S 8]{PrasadD1993}).

\begin{corollary}\label{cgtenHsim}
Let $\pi=\otimes_v\pi_v$ be an automorphic representation of $G$.
Assume that the local big theta lift is irreducible at every place, i.e.\ $\Theta(\pi_v)=\theta(\pi_v)$ for all $v$.
Under the assumption of Theorem~\ref{tlgcomp}, we have
\[
V(\pi)\cong \pi^{\vee}\otimes_{\mathcal{H}(G)}\Omega_{\mathbb{A}}
\]
as a representation of $H(\mathbb{A})$.
\end{corollary}

\begin{proof}
This follows from Lemma~\ref{ltenHsim} and Theorem~\ref{tlgcomp}.
\end{proof}

\section{The tensor product for the global Jacquet--Langlands correspondence}\label{stenJL}

In this section, we fix a quaternion division algebra $D$ over a number field $F$.
Let $\nrd$ be the reduced norm defined by $\nrd(x)=x\cdot \overline{x}$ and $\tr_{D/F}$ be the trace given by $\tr(x)=x+\overline{x}$.

Fix a non-archimedean place $v$ of $F$ and consider the quaternionic case where $V=V_1$ and $W=W_2$ as described in Example~\ref{eGTquasim}.
We have $V_1\cong D$ and $W_2=D f_1\oplus D f_2$ with $(f_1,f_2)=1$ and $(f_i,f_i)=0$.
Thus, the Weil representation of $\U(V_1)\times \U(W_2)$ can be realized on the space $\mathcal{S}(D)$ of Schwartz--Bruhat functions on $D$ as follows.
\begin{enumerate}
    \item \underline{$\U(V_1)$}.
    The isometry group and similitude group are
    \begin{center}
        $\U(V_1)\cong D^1$ and $\GU(V_1)\cong D^{\times}$.
    \end{center}
    For $g\in \U(V_1)$ and $\phi\in \mathcal{S}(D)$, we have
    \[
        (\omega_{\psi}(g)\phi)(x)=\phi(g^{-1}x).
    \]

    \item \underline{$\U(W_2)$}.
    Recall that $\U(W_2)$ consists of elements in $\End_{D}(W_2)$ that preserve the inner product $(~,~)$.
    With respect to the basis $f_1,f_2$, we may represent $h\in \U(W_2)$ by
    $\begin{pmatrix}
h_{1,1} & h_{1,2} \\
h_{2,1} & h_{2,2}
\end{pmatrix}$ with $h_{i,j}\in D$.
We can further show that the $\U(W_2)$-action on $W_2$ can be represented by elements of the form
    $(d,\begin{pmatrix}
a_{1,1} & a_{1,2} \\
a_{2,1} & a_{2,2}
\end{pmatrix})\in D^{\times}\times \GL(2,F_{v})$, acting by
\begin{center}
    $(d,\begin{pmatrix}
a_{1,1} & a_{1,2} \\
a_{2,1} & a_{2,2}
\end{pmatrix})f_1 = d(a_{1,1}f_{1}+a_{1,2}f_{2})$,\\
    $(d,\begin{pmatrix}
a_{1,1} & a_{1,2} \\
a_{2,1} & a_{2,2}
\end{pmatrix})f_2 = d(a_{2,1}f_{1}+a_{2,2}f_{2})$.
\end{center}
We obtain
\[
\U(W_2)\cong \{(d,a)\in D^{\times}\times \GL(2,F_{v})\mid \nrd(d)\det(a)=1\}/\{(z,z^{-1})\mid z\in F_{v}^{\times}\},
\]
and
\[
\GU(W_2)=\GO_{1,1}^{*}\cong (D^{\times}\times \GL(2,F_{v}))/\{(t,t^{-1})\mid t\in F_{v}^{\times}\}.
\]
Since elements of the form
$\begin{pmatrix}
a & 0 \\
0 & \overline{a}^{-1}
\end{pmatrix}$,
$\begin{pmatrix}
1 & n \\
0 & 1
\end{pmatrix}$, and
$\begin{pmatrix}
0 & -1 \\
1 & 0
\end{pmatrix}$
($a\in D^{\times}$ and $n\in D$) generate $\U(W_2)$ by the Bruhat decomposition, it suffices to define their actions.
\begin{enumerate}
\item $(\omega_{\psi}(\begin{pmatrix}
a & 0 \\
0 & \overline{a}^{-1}
\end{pmatrix})\phi)(x)=\nrd(aa^*)^{1/2}\phi(xa)$;
\item $(\omega_{\psi}(\begin{pmatrix}
1 & n \\
0 & 1
\end{pmatrix})\phi)(x)=\psi(\frac{\tr_{D/F}(\nrd(x)n)}{2})\phi(x)$;
\item $(\omega_{\psi}(\begin{pmatrix}
0 & -1 \\
1 & 0
\end{pmatrix})\phi)(x)=\lambda_{\psi}\int_{D}\phi(y)\psi(\tr_{D/F}(x\overline{y}))\dd y$,
where $\lambda_{\psi}$ is a constant depending on $\psi$ and the Haar measure on $D$.
\end{enumerate}
\end{enumerate}

Now we give an explicit description of the underlying space of the induced Weil representation $\Omega$ for the associated similitude dual pair (see \eqref{eindWeil}) using a Schwartz--Bruhat space.
We fix the following notation for the similitude dual pair:
\begin{equation*}
\begin{aligned}
G=\GU(V_1)=\GSp_{1,0}&\cong D^{\times},\\
H=\GU(W_2)=\GO_{1,1}^{*}&\cong (D^{\times}\times \GL(2,F))/\{(t,t^{-1})\mid t\in F^{\times}\}.
\end{aligned}
\end{equation*}
By Proposition~\ref{ptwoOmega}, $\Omega$ acts on the space
\[
    V_{\Omega}=\{f\colon G\to \mathcal{S}(D)\mid f(g_1 g)=\omega(g_1)f(g)\ \text{for all } g_1\in G_1,\ g\in G\}.
\]
Define the linear map
\[
    \Phi\colon V_{\Omega}\to \mathcal{S}(D\times F^{\times})
\]
by
\[
    \phi_{f}(x,y)=f(d(y)^{-1}x),
\]
for $f\in V_{\Omega}$ and $(x,y)\in D\times F^{\times}$.

For $t\in F_{v}^{\times}$, we let $d(t)$ be the similitude factor belongs to $G$ or $H$ such that $\lambda(d(t)=t$. 
The $G\times H$-action can be described as follows.
\begin{enumerate}
\item For $g\in G_1=D^{1}$, we let
\begin{center}
    $\Omega(g,1)\phi(x,y)=|\lambda_{G}(g)|^{1/2}\phi(g^{-1}x,\lambda_{G}(g)y)$,
\end{center}
where $g^{-1}x$ denotes the left translation by $g^{-1}$ on $x\in D$. 
\item For $h\in \U(W_2)$, we let
\begin{center}
    $\Omega(1,h)\phi(x,y)=(\omega_{\psi}(1,h^t)\phi_{t})(x)$,
\end{center}
where $h^t$ denotes the conjugate of $h$ by a similitude factor $d(t)$ and $\phi_{t}\in \mathcal{S}(D)$ is given by $\phi_{t}(x)=\phi(x,t)$. 
\item For a similitude factor $d(s)\in H$ with $s\in F_{v}^{\times}$, we let
\begin{center}
    $(\Omega(1,d(s))\phi)(x,t)=\phi(x,s^{-1}t)$. 
\end{center}
\item For a general $h\in H$, we may write it as $h=d(s)h_1$ with $h_1\in \U(W_2)$ and $d(s)$ denotes the similitude factor such that $s=\lambda_{H}(h)$. 
We let
\begin{center}
 $\Omega(1,h)=\Omega(1,d(s))\Omega(1,h_1)$.
\end{center}
\end{enumerate}
Therefore, we obtain a concrete model for $\Omega$ as the symplectic-orthogonal case in \cite[Section 3]{RobertBrooks1996}. 

\begin{lemma}\label{lOmegaS}
The map $\Phi\colon V_{\Omega}\to \mathcal{S}(D\times F^{\times})$ is a $G\times H$-equivariant isomorphism.
\end{lemma}

We next record the irreducibility of the similitude big theta lift for irreducible representations of the local groups $G(F_v)$.

\begin{proposition}\label{psimstb}
With the similitude dual pair $(G,H)$ above, for each place $v$ and each $\pi_v\in \irr(G(F_v))$, we have
\[
\Theta(\pi_v)\cong\theta(\pi_v),
\]
and therefore $\Theta(\pi_v)$ is irreducible.
\end{proposition}

\begin{proof}
First, consider the dual pair of isometry groups $(G_1,H_1)$.
For an irreducible representation $\pi_v$ of $G(F_v)$, we may write
\[
\pi_v|_{G_1(F_v)}\cong \bigoplus_{i\in I}\sigma_i
\]
as a decomposition into irreducible representations of $G_{1}(F_v)$, where $I$ is finite.
Note that $(G,H)$ is a dual pair in the stable range (for the isometry subgroups), with $G$ the smaller member.
By \cite[Theorem~A]{LokeMa2015} and \cite[Corollary~7.3]{ChenZou2024}, the big theta lift is irreducible:
$\Theta_{\psi}(\sigma_i)=\theta_{\psi}(\sigma_i)$ for $\sigma_i\in \irr(G_{1}(F_v))$, both in the archimedean and the non-archimedean cases.
By \cite[Lemma~2.2(iii)]{GanTak2011}, we conclude that $\Theta(\pi_v)=\theta(\pi_v)$.
\end{proof}

We briefly review the Jacquet--Langlands correspondence, initiated by Jacquet and Langlands in \cite{JL70}.
We follow a recent approach of Badulescu (see \cite{Badu08}).

Let $Z$ be the center of $G$, so $Z\cong \mathbb{G}_m$.
For a unitary character $\chi\colon Z(\mathbb{A})\to \mathbb{C}^{\times}$ that is trivial on $Z(F)$, define $L^2(G(F)\backslash G(\mathbb{A}),\chi)$ to be the space of functions $f\colon G(\mathbb{A})\to \mathbb{C}$ such that
\[
f(\gamma z g)=\chi(z)f(g)\quad \text{for } z\in Z(\mathbb{A}),\ \gamma\in G(F),\ g\in G(\mathbb{A}),
\]
and
\[
\int_{G(F)Z(\mathbb{A})\backslash G(\mathbb{A})}|f(g)|^2\,\dd g<\infty.
\]
Let $R_{\chi}$ denote the quasi-regular representation of $G(\mathbb{A})$ on $L^2(G(F)\backslash G(\mathbb{A}),\chi)$.
Since the quotient $G(F)Z(\mathbb{A})\backslash G(\mathbb{A})$ is compact when $D$ is non-split, $R_{\chi}$ decomposes discretely into irreducible representations of $G(\mathbb{A})$.
Let $\irr_{\mathcal{A}}(G,\chi)$ be the set of equivalence classes of irreducible automorphic representations of $G(\mathbb{A})$, i.e.\ the irreducible subrepresentations of $R_{\chi}$.

For the group $\GL(2)$, the analogous quasi-regular representation decomposes into a discrete and a continuous spectrum.
Let $\irr_{\mathcal{A},\disc}(\GL(2),\chi)$ be the set of equivalence classes of irreducible representations of $\GL(2,\mathbb{A})$ contained in the discrete spectrum with central character $\chi$.
By \cite[Theorem~5.1]{Badu08}, there exist two injective maps:
\begin{itemize}
\item the \emph{local Jacquet--Langlands correspondence} $\JL_{v}\colon \irr(G(F_v))\to \irr(\GL(2,F_v))$ for each place $v$;
\item the \emph{global Jacquet--Langlands correspondence} $\JL\colon \irr_{\mathcal{A}}(G,\chi)\to\irr_{\mathcal{A},\disc}(\GL(2),\chi)$ such that:
if $\JL(\otimes_{v}\pi_v)=\otimes_{v}\pi'_v$, then $\JL_{v}(\pi_v)=\pi'_{v}$ for each place $v$.
\end{itemize}
It is known that $G(F_v)$ and $\GL(2,F_v)$ are isomorphic for all but finitely many places, and $\JL_v$ is the identity at these places.

The global Jacquet--Langlands correspondence is usually formulated for unitary representations.
Thus, for a unitarizable smooth representation $\pi$, we write $\pi$ also for its unitary completion, and we replace the contragredient $\pi^{\vee}$ by the Hilbert space dual $\pi^*$ when working in the unitary category.
For the global induced Weil representation $\Omega_{\mathbb{A}}$ defined in \eqref{eglobalOmega}, it is known that
\[
\Omega_{\mathbb{A}}\cong \mathcal{S}(D(\mathbb{A})\times \mathbb{A}^{\times})
\]
by Lemma~\ref{lOmegaS}.
Thus, for its completion $\overline{\Omega_{\mathbb{A}}}$, we have
\[
\overline{\Omega_{\mathbb{A}}}\cong L^2(D(\mathbb{A})\times \mathbb{A}^{\times}).
\]
Recall that $\mathcal{H}(G)$ denotes the global Hecke algebra defined in \eqref{egHecke}.

\begin{theorem}\label{tJltenHecke}
Let $G=D^{\times}$ and let $\mathcal{H}(G)$ be the global Hecke algebra of $G$ defined above.
Let $\pi$ be an irreducible automorphic representation of $G(\mathbb{A})$ with central character $\chi$.
Then
\begin{equation}
   \pi^{\vee}\otimes_{\mathcal{H}(G)}\mathcal{S}(D(\mathbb{A})\times \mathbb{A}^{\times})\cong\pi\otimes\JL(\pi)
\end{equation}
as smooth representations of $D^{\times}(\mathbb{A})\times \GL(2,\mathbb{A})$.
In particular,
\begin{equation}\label{etmain}
    L^2(G(F)\backslash G(\mathbb{A}),\chi)\otimes_{\mathcal{H}(G)}L^2(D(\mathbb{A})\times \mathbb{A}^{\times})
\cong
\bigoplus_{\pi\in\irr_{\mathcal{A}}(G,\chi^{-1})}\pi\otimes\JL(\pi).
\end{equation}
\end{theorem}

\begin{proof}
Take $\pi\in \irr_{\mathcal{A}}(G,\chi)$ and consider its big theta lift $V(\pi)$ to $H$.
Since the pair $(G,H)$ (or, equivalently, the pair of isometry groups $(\U(V_1),\U(W_2))$) is a dual pair in the stable range with $G$ the smaller member, $V(\pi)$ is nonzero and cuspidal by \cite[Proposition~3.2]{Gan2024} (see also \cite[Theorem~I.1.1]{Rallis1984}).

By Theorem~\ref{tlgcomp}, we have $V(\pi)\cong\otimes_{v}\theta(\pi_v)$.
Note that $G(F_v)$ is split for all but finitely many places $v$ of $F$.
For such a place $v$, we have
\[
H(F_v)\cong (\GL(2,F_v)\times \GL(2,F_v))/\{(t,t^{-1})\mid t\in F_v^{\times}\}.
\]
Thus, the local theta lift gives $\theta(\pi_v)\cong \pi_v\otimes\pi_v$ at split places.

Assume $\theta(\pi)=\pi'\otimes\pi''$ for cuspidal representations $\pi',\pi''$ of $D^{\times}(\mathbb{A})$ and $\GL(2,\mathbb{A})$, respectively.
Then $\pi'_v\cong\pi_v$ and $\pi''_v\cong \pi_v$ at split places.
By strong multiplicity one theorem \cite{Shapiro1979mult1} and the global Jacquet--Langlands correspondence \cite[Theorem~5.1]{Badu08}, we obtain
\[
\pi'_v\cong \pi_v\quad \text{and}\quad \pi''_v\cong \JL_v(\pi_v)\quad \text{for all } v.
\]
Equivalently, $\theta(\pi)\cong \pi\otimes \JL(\pi)$.

By Proposition~\ref{psimstb}, we have $\Theta(\pi_v)=\theta(\pi_v)$ for each $v$.
Then Corollary~\ref{cgtenHsim} yields $\pi^{\vee}\otimes_{\mathcal{H}(G)}\Omega_{\mathbb{A}}\cong\pi\otimes\JL(\pi)$.
The final decomposition follows from the multiplicity-one theorem for $D^{\times}(\mathbb{A})$ (see \cite[Theorem~5.1]{Badu08}). 
\end{proof}

\begin{remark}
Observe that the space $L^2(G(F)\backslash G(\mathbb{A}))$ decomposes with respect to central characters as
\[
L^2(G(F)\backslash G(\mathbb{A}))=\int_{\widehat{Z(F)\backslash Z(\mathbb{A})}}^{\oplus}L^2(G(F)\backslash G(\mathbb{A}),\chi)\,\dd\chi.
\]
Then Theorem~\ref{tJltenHecke} implies
\[
L^2(G(F)\backslash G(\mathbb{A}))\otimes_{\mathcal{H}(G)}L^2(D(\mathbb{A})\times \mathbb{A}^{\times})
=
\int_{\widehat{Z(F)\backslash Z(\mathbb{A})}}^{\oplus}\left(\bigoplus_{\pi\in\irr_{\mathcal{A}}(G,\chi)}\pi\otimes\JL(\pi)\right)\dd\chi.
\]
\end{remark}

\apptocmd{\thebibliography}{%
  \setlength{\itemsep}{2pt}%
  \setlength{\parsep}{0pt}%
  \setlength{\parskip}{0pt}%
}{}{}

\bibliographystyle{abbrv}
\typeout{}
\bibliography{MyLibrary} 

\noindent
\textit{E-mail address}: \href{mailto:junyang@fas.harvard.edu}{junyang@fas.harvard.edu}\\
\noindent
\textit{Address}:{ Harvard University, Cambridge, MA 02138, USA}

\end{document}